\documentclass{amsart}
\usepackage{pgf,tikz, graphics, latexsym}
\usepackage{graphicx}
\usepackage{float}
\usepackage{setspace}
\usepackage{pifont}
\usepackage[colorlinks=true,linkcolor=blue,citecolor=blue]{hyperref}

\usepackage{upref,amsmath,amssymb,amsthm,mathrsfs,amsfonts,mathtools}

\usetikzlibrary{arrows}
\newtheorem{theorem}{Theorem}[section]
\newtheorem*{theorem*}{Theorem}
\newtheorem{corollary}[theorem]{Corollary}
\newtheorem{lemma}[theorem]{Lemma}
\newtheorem{proposition}[theorem]{Proposition}
\theoremstyle{definition}
\newtheorem{definition}[theorem]{Definition}
\theoremstyle{remark}
\newtheorem{remark}[theorem]{Remark}

\numberwithin{equation}{section}

\newcommand{\C}{\mathbb{C}}

\newcommand{\alfa}{\mathbb{A}}

\newcommand{\es}{\mathcal{S}}
\newcommand{\ti}{\mathcal{T}}

\newcommand{\closure}{{-w^*}}
\newcommand{\normcl}{{-\|\cdot\|}}

\newcommand{\tro}{\mathcal{M}}
\newcommand{\ntro}{\mathcal{N}}

\newcommand{\emm}{\textbf{m}}
\newcommand{\en}{\textbf{n}}
\newcommand{\spacex}{\mathcal{X}}
\newcommand{\spacey}{\mathcal{Y}}
\newcommand{\kk}{\mathcal{K}}
\newcommand{\nsp}{\bar{\otimes}}
\newcommand{\lat}{\text{Lat}}

\title{Morita equivalence and stable isomorphism via unitary operators}

\author{Nikolaos Koutsonikos-Kouloumpis}

\address{
University of Patras\\
Department of Mathematics\\
26504 Patras\\
Greece}

\email{up1019669@ac.upatras.gr}

\date{}



\begin{document}

\begin{abstract}
    We define $\Delta$-equivalence for dual operator systems and prove that it is an equivalence relation. We show that weak TRO-equivalence of dual operator spaces induces a stable isomorphism between them which is given by multiplication with unitary operators, and in the special case of dual operator systems it is a unitary equivalence. We prove an analogous result for strong TRO-equivalence of operator spaces and for operator systems. Lastly, we show that $\Delta$-equivalent dual operator spaces, considered as bimodules over their left and right adjointable multiplier algebras, have TRO-equivalent normal CES representations. 
\end{abstract}

\maketitle

\section{Introduction}

In the 1970's, M. Rieffel \cite{rieffel} introduced to the world of noncommutative analysis the concept of \textit{Morita equivalence}, in a strong version as an equivalence of $C^*$-algebras and in a weak version as an equivalence of $W^*$-algebras. Rieffel was able to transfer most of the results in the theory of Morita equivalence of rings \cite{morita, bass} to the categories of $C^*$- and $W^*$-algebras. 

Following Rieffel, as well as the development of operator space theory \cite{effros}, D. P. Blecher, P. S. Muhly and V. I. Paulsen \cite{bmp} defined an analogue of Morita equivalence for non-selfadjoint operator algebras. Later, D. P. Blecher and U. Kashyap \cite{blkas} defined Morita equivalence for dual operator algebras and derived analogous results.

The idea of introducing \textit{ternary rings of operators} (TRO's) to the theory of Morita equivalence led to a significant progress. The notion of TRO-equivalence \cite{troeq} gave rise to $\Delta$-equivalence, a Morita type equivalence that forms a generalization of Morita equivalence of $W^*$-algebras, to (possibly non-selfadjoint) dual operator algebras. The definition of $\Delta$-equivalence was introduced by G. K. Eleftherakis in \cite{dualopalg}, and then an analogous definition was given for dual operator spaces \cite{stable}. Strong counterparts of $\Delta$-equivalence for operator algebras \cite{ele}, then for operator spaces \cite{strspaces} and finally for operator systems \cite{systems} were defined too, as a generalization of strong Morita equivalence for $C^*$-algebras.

A great demonstration of the significance of Morita equivalence is its close relation to \textit{stable isomorphisms}. Stable isomorphism of $C^*$-algebras was introduced by L. G. Brown, P. Green and M. A. Rieffel \cite{moritacstar}, where the authors called two $C^*$-algebras stably isomorphic in case the $C^*$-algebras formed by tensoring with the $C^*$-algebra of the compact operators acting on a separable Hilbert space are isomorphic as $C^*$-algebras (similar definitions hold for operator algebras, spaces etc.). It was shown in \cite{moritacstar} that two $C^*$-algebras that possess strictly positive elements (or equivalently a countable approximate identity) are Morita equivalent if and only if they are stably isomorphic. G. K. Eleftherakis showed \cite{ele} that two operator algebras with countable approximate identities are $\Delta$-equivalent if and only if they are stably isomorphic and the same was shown in \cite{strspaces} for operator spaces, with a fitting condition that resembles the existence of a countable approximate identity in the forementioned cases. Finally, in \cite{stablemaps} G. K. Eleftherakis introduced $\sigma\Delta$-equivalence and showed that two operator algebras are $\sigma\Delta$-equivalent if and only if they are stably isomorphic.

In the dual picture, two $W^*$-algebras are Morita equivalent if and only if they are stably isomorphic, in the sense that the $W^*$-algebras formed by tensoring with $B(H)$ for some Hilbert space $H$ are $*$-isomorphic via a $w^*$-homeomorphism. The cases of dual operator algebras and dual operator spaces were handled by G. K. Eleftherakis and V. I. Paulsen \cite{stalgebras}, accompanied by I. G. Todorov in the latter case \cite{stable}, who showed the equivalence of stable isomorphism with $\Delta$-equivalence in each particular category.

In this paper, we define $\Delta$-equivalence and stable isomorphism for \textit{dual operator systems}, objects which where introduced by D. P. Blecher and B. Magajna \cite{dualopsys}, and we prove some of the main results of Morita theory in this category's context. We establish the fact that $\Delta$-equivalence implies stable isomorphism and vice versa, as in the cases of dual operator algebras and dual operator spaces (Corollary \ref{dualopsysstable}).

One of the main results of this note however, is an enhancement of the relation between stable isomorphism and TRO-equivalence. As it turns out, TRO-equivalent spaces are not just stably isomorphic, but the isomorphism constructed is given by multiplication with unitary operators, whereas in the case of operator systems it is a unitary equivalence. We show this both for weakly TRO-equivalent dual operator spaces (Theorem \ref{uni}) and for strongly TRO-equivalent operator spaces (Theorem \ref{opspc}) and the case of (dual) operator systems is proved similarly (Theorem \ref{unisys}, Theorem \ref{strsys}).

We begin in Section \ref{prel} by collecting some preliminaries and fixing some notation and then in Section \ref{sec:delta}, we define TRO- and $\Delta$-equivalence for dual operator systems and we show that it is in fact an equivalence relation. In Section \ref{sec:weakstable} we deal with weak $\Delta$-equivalence of dual operator spaces and dual operator systems. We show that TRO-equivalent nondegenerate $w^*$-closed concrete operator spaces (resp. systems) are stably isomorphic via unitary operators (resp. unitary equivalence). Moreover, we use this in order to show that two $\Delta$-equivalent dual operator spaces, considered as dual operator bimodules over their left and right adjointable multiplier algebras $A_\ell(\cdot)$ and $A_r(\cdot)$ admit CES representations, with TRO-equivalent images (Proposition \ref{CEStheorem}) and an analogous result (Proposition \ref{CESsystems}) holds for dual operator systems, too. Finally, in Section \ref{sec:strongstable} we show how strong TRO-equivalence coincides with strong stable isomorphism via unitary operators under the assumption that there exists an appropriate countable approximate identity.

\section{Preliminaries} \label{prel}

For Hilbert spaces $H,K$ we write $B(H,K)$ for the set of all linear bounded operators from $H$ into $K$; as usual we set $B(H):=B(H,H)$. The identity operator in $B(H)$ is denoted by $I_H$ (or just $I$ when the Hilbert space $H$ is implied). By $H\otimes K$ we denote the Hilbert space tensor product of $H$ with $K$. For a cardinal $J$, we set $\ell_J^2$, the Hilbert space of square-summable sequences indexed by the set $J$

For any set $A\subseteq B(H,K)$, we write $\text{alg}_\ell(A)$ and $\text{alg}_r(A)$ for the algebras
\[
\{T\in B(K):TA\subseteq A\} \ \ \text{and} \ \ \{T\in B(H):AT\subseteq A\}
\]respectively, and for an algebra $B\subseteq B(H)$ we set $\Delta(B):=B\cap B^*$.

Let $H,K,L$ be Hilbert spaces and $X,Y$ subspaces of $B(H,K)$ and $B(K,L)$ respectively. We write $\overline{[XH]}$ for the closed linear span of all elements of the form $x\xi$ for $x\in X$ and $\xi\in H$ and $[YX]^\normcl$ (resp. $[YX]^\closure$) for the norm-closed (resp. $w^*$-closed) linear span of all elements of the form $yx$ for $y\in Y$, $x\in X$. We say that $X$ \textit{acts nondegenerately} if $\overline{[XH]}=K$ and $\overline{[X^*K]}=H$. A subspace $\tro\subseteq B(H,K)$ is called a \textit{ternary ring of operators} (TRO) if $\tro\tro^*\tro\subseteq \tro$.

For a vector space $V$, $M_{m,n}(V)$ denotes the set of all $m\times n$ matrices with elements in $V$, with the obvious linear structure. If $\phi:V\to W$ is a linear map, its $n$-th \textit{amplification} $\phi^{(n)}:M_n(V)\to M_n(W)$ is the map $(v_{i,j})_{i,j=1}^n\mapsto (\phi(v_{i,j}))_{i,j}^{n}$. For cardinals $J_1,J_2$, $M_{J_1,J_2}$ denotes the space of all $J_1\times J_2$ matrices with elements in $\C$, whose finite submatrices have uniformly bounded norms. The norm of a matrix in $M_{J_1, J_2}$ is set to be the supremum of norms of all its finite submatrices and we have $M_{J_1,J_2}\cong B(\ell_{J_2}^2,\ell_{J_1}^2)$ $*$-isomorphically. The norm closure of the space of all matrices in $M_{J_1,J_2}$ with all but finitely many entries equal to zero is denoted by $\kk_{J_1,J_2}$, which is identified with the set of all compact operators in $B(\ell_{J_2}^2,\ell^2_{J_1})$ and we set $M_J:=M_{J, J}$, $\kk_J:=\kk_{J,J}$, $C_J:=\kk_{J,1}=M_{J,1}$ and $R_J:=\kk_{1,J}=M_{1,J}$. We use the notation $\kk_\infty$, $M_\infty$ etc. for $J=\mathbb{N}$.

A \textit{concrete unital operator system} (or just \textit{concrete operator system}) is a norm-closed subspace $\es\subseteq B(H)$ that is self-adjoint (i.e. $\es^*=\es$) and contains the unit $I_H$. We refer the reader to \cite[Chapter 13]{paulsen} for the definition of an \textit{(abstract) unital operator system}, or just \textit{operator system}. A linear map $\phi:\es\to \ti$ between operator systems is called \textit{positive} if $\phi(s)\geq 0$ for all $s\geq 0$, and it is called \textit{completely positive} if $\phi^{(n)}$ is positive for all $n$. An injective map $\phi:\es\to \ti$ between operator systems is called a \textit{complete order embedding} if $s\geq 0$ if and only if $\phi^{(n)}(s)\geq 0$ for all $n\in\mathbb{N}$ and $s\in M_n(\es)$. 

For an operator space $X$, any bounded linear functional $\phi:X\to \mathbb{C}$ is completely bounded, thus we have $X^*=CB(X,\mathbb{C})$, the space of all completely bounded linear functionals, which has the structure of an operator space (see e.g. \cite{blecher} or \cite{effros}). We call such a space a \textit{dual operator space}. As a Banach space dual, every dual operator space carries a $w^*$-topology, and by \cite[Proposition 3.2.4]{effros}, it can be viewed as a $w^*$-closed subspace of $B(H)$ for some Hilbert space $H$. 

For two operator spaces $X,Y$, we denote by $X\otimes Y$ their \textit{minimal tensor product}. In case $X\subseteq B(H_1,H_2)$ and $Y\subseteq B(K_1,K_2)$, $X\otimes Y$ is completely isometric to the closed linear span in $B(H_1\otimes K_1,H_2\otimes K_2)$ of all elements of the form $x\otimes y$, $x\in X$, $y\in Y$. If $X,Y$ are dual operator spaces, we denote by $X\nsp Y$ their \textit{normal spatial tensor product}. In case $X\subseteq B(H_1,H_2)$ and $Y\subseteq B(K_1,K_2)$ are $w^*$-closed subspaces, $X\nsp Y$ is $w^*$-homeomorphically completely isometric to the $w^*$-closed linear span in $B(H_1\otimes K_1,H_2\otimes K_2)$ of all elements of the form $x\otimes y$, $x\in X$, $y\in Y$. For more details, the reader is referred to \cite[Sections 1.5, 1.6]{blecher}.

A \textit{dual operator system} \cite{dualopsys} is an operator system that is also a dual operator space. By \cite[Theorem 1.1]{dualopsys} any dual operator system may be realized as a $w^*$-closed operator subsystem of $B(H)$ for a Hilbert space $H$. 

If $A,B$ are approximately unital operator algebras, an \textit{operator $A-B$-bimodule} is an operator space $X$ equipped with completely contractive bilinear maps $A\times X\to X$, $X\times B\to X$. An operator $A-B$-bimodule is called \textit{nondegenerate} if $[AX]^\normcl=[XB]^\normcl=X$. If $A,B$ are unital dual operator algebras and $X$ is an operator space, with the additional assumption that the module actions $A\times X\to X$, $X\times B\to X$ are separately $w^*$-continuous, $X$ is called a \textit{dual operator $A-B$-bimodule}, and in this case it is called nondegenerate if $[AX]^\closure=[XB]^\closure=X$. 

If $X$ is an operator $A-B$-bimodule, a \textit{CES representation} for $X$ is a triple $(\pi,\phi,\sigma)$ of maps, where $\pi:A\to B(K)$ and $\sigma:B\to B(H)$ are completely contractive representations of $A,B$ on the Hilbert spaces $K$ and $H$ respectively, and $\phi:X\to B(H,K)$ is a linear complete isometry such that
\[
\phi(axb)=\pi(a)\phi(x)\sigma(b)
\]for all $a\in A$, $x\in X$, $b\in B$. In case $\pi,\sigma$ are completely isometric, the CES representation is called \textit{faithful}. By \cite[Theorem 3.3.1]{blecher}, every nondegenerate $A-B$-bimodule admits a faithful CES representation. 

If $X$ is a dual operator $A-B$-bimodule, a \textit{normal CES representation} for $X$ is a triple $(\pi,\phi,\sigma)$ of maps as before, so that they are all $w^*$-continuous and $\pi:A\to B(K)$, $\sigma:B\to B(H)$ are unital. In case $\pi,\sigma$ are completely isometric, it is called a \textit{faithful normal CES representation}. By \cite[Theorem 3.8.3]{blecher}, every dual operator $A-B$-bimodule admits a faithful normal CES representation.

For the definitions of \textit{left} (resp. \textit{right}) \textit{multipliers} of an operator space $X$ and the corresponding \textit{left} (resp. \textit{right}) \textit{multiplier algebra}, $M_\ell(X)$ (resp. $M_r(X)$), as well as the \textit{left} (resp. \textit{right}) \textit{adjointable multiplier algebra} $A_\ell(X)$ (resp. $A_r(X)$), the reader is referred to \cite[Section 4.5]{blecher}. We note that by \cite[Theorem 4.5.5]{blecher} (see also \cite[4.6.6]{blecher}), every operator space $X$ is a $M_\ell(X)-M_r(X)$-bimodule, as well as a $A_\ell(X)-A_r(X)$-bimodule.

When $X$ is a dual operator space, we have the following: 
\begin{theorem}{\cite[Theorems 4.7.1, 4.7.4]{blecher}} \label{magajna}
    If $X$ is a dual operator space, every map in $M_\ell(X)$ (thus every map in $A_\ell(X)$ too) is $w^*$-continuous. Furthermore, $M_\ell(X)$ is a dual operator algebra and $A_\ell(X)$ is a $W^*$-algebra and for every bounded net $(u_\lambda)\subseteq M_\ell(X)$ and $u\in M_\ell(X)$ (or $A_\ell(X)$), $u_\lambda$ converges to $u$ in the $w^*$-topology if and only if $u_\lambda(x)\to u(x)$ in the $w^*$-topology for each $x\in X$. 
\end{theorem} 
\noindent It follows directly in this case that $X$ is a dual operator $M_\ell(X)-M_r(X)$- and $A_\ell(X)-A_r(X)$-bimodule.

In case $\es$ is an operator system, $\es$ can be viewed as a subsystem of its $C^*$-envelope \cite{envelope}, $C_e^*(\es)$ and 
\begin{equation}
\label{leftm}
    M_\ell(\es)\cong\{a\in C_e^*(\es):a\es\subseteq \es\}
\end{equation}
as operator algebras, by associating an element $a$ in the right hand side to the left multiplier $s\mapsto as$. It is easy to see, using \eqref{leftm} that $M_\ell(\es)$ can be embedded unitally completely isometrically into $\es$ via the map $u\mapsto u(1)$. Moreover, 
\begin{equation}
\label{adjointable}
    A_\ell(\es)\cong \{a\in C_e^*(\es):a\es\subseteq \es, a^*\es\subseteq \es\}
\end{equation}
as $C^*$-algebras. By \eqref{adjointable}, we can embed $A_\ell(\es)$ into $\es$ via the map $u\mapsto u(1)$, which is a unital complete isometry, or equivalently a unital complete order embedding. By \eqref{adjointable}, it is clear that $A_\ell(\es)\cong A_r(\es)$ as $C^*$-algebras when $\es$ is an operator system (note that $A_r(\es)$ is equipped with the reverse structure of the usual composition multiplication). Writing $\iota_\es:A_\ell(\es)\to \es$ for the forementioned embedding, we will denote $\iota_\es(A_\ell(\es))\subseteq \es$ by $\alfa_\es$ (we will treat elements of $A_\ell(\es)$ as maps $u:\es\to \es$, whereas we think of $\alfa_\es$ as a $C^*$-subalgebra of $C_e^*(\es)$ contained in $\es$). Then, one can easily check that every operator system $\es$ is an operator $\alfa_\es$-system in the sense of \cite[Chapter 15]{paulsen}.

In case $\es$ is a dual operator system, thanks to the latter fact, combined with the Krein-Smulian Theorem, the inclusion map $\iota_\es:A_\ell(\es)\to \es$ is $w^*$-continuous, hence a $w^*$-homeomorphism onto its image, $\alfa_\es$. However, $C^*_e(\es)$ is not expected to carry a dual structure compatible with the dual structure of $\es$, so in order to keep track of the $w^*$-topology of $\es$, one has to be careful in the use of the $C^*$-envelope. It is clear by Theorem \ref{magajna} though, that the module action from $\alfa_\es\times \es$ into $\es$, mapping $(a,s)$ to $as$ is separately $w^*$-continuous.

We recall by \cite{lin} that for a $W^*$-algebra $\mathcal{A}$ and an operator system $X$, $X$ is called a \textit{dual operator $\mathcal{A}$-system} if:
\begin{itemize}
    \item[(i)] $X$ is an operator $\mathcal{A}$-system,
    \item[(ii)] $X$ is a dual operator system, and
    \item[(iii)] the map $(a,x)\mapsto a\cdot x$ from ${\mathcal{A}}\times X$ into $X$ is separately $w^*$-continuous. 
\end{itemize}

\begin{theorem}{\cite[Theorem 4.7]{lin}}
\label{rep}
    Suppose $\mathcal{A}$ is a $W^*$-algebra and $X$ a dual operator $\mathcal{A}$-system. There exists a Hilbert space $H$, a normal unital complete order embedding $\gamma: X\to B(H)$ and a unital $*$-homomorphism $\pi:{\mathcal{A}}\to B(H)$ such that 
    \[
    \gamma(a\cdot x)=\pi(a)\gamma(x)
    \]for all $a\in \mathcal{A}$, $x\in X$.
    
\end{theorem}

\begin{remark}
\label{rep2}
    It is clear by the discussion in the previous paragraph that a dual operator system $\es$ is a dual operator $\alfa_\es$-system. Let $H,\gamma,\pi$ be as in Theorem \ref{rep}. Since $\alfa_\es$ embeds completely isometrically into $\es$, we get that $\pi$ is also a complete isometry, and it is in fact the restriction of $\gamma$ to $\alfa_\es$. 
\end{remark}

\section{\texorpdfstring{$\Delta$-equivalence}{Delta-equivalence}} \label{sec:delta}

We start by investigating the nature of $\Delta$-equivalence for dual operator systems and its main properties. 

\begin{definition}
    Two $w^*$-closed operator subsystems $\es\subseteq B(H)$ and $\ti\subseteq B(K)$ are called \textit{TRO-equivalent} in case there exists a TRO $\tro\subseteq B(H,K)$ such that $[\tro\es\tro^*]^\closure=\ti$ and $[\tro^*\ti\tro]^\closure=\es$. In that case we write $\es\sim_{TRO}\ti$ or $\es\sim_\tro\ti$ in case we want to specify the TRO $\tro$ which implements the equivalence.
\end{definition}

\begin{remark}
    If $\es\sim_{\tro}\ti$, then $\tro$ acts nondegenerately because of the fact that $\es,\ti$ are unital, therefore $[\tro^*\tro]^\closure$ and $[\tro\tro^*]^\closure$ are unital and they are von Neumann subalgebras of $B(H)$ and $B(K)$ respectively. Furthermore, $\es$ (resp. $\ti$) is a bimodule over $[\tro^*\tro]^\closure$ (resp. $[\tro\tro^*]^\closure$). 
\end{remark}

\begin{remark}
\label{nondeg}
    In case $\tro$ acts nondegenerately and $\tro\es\tro^*\subseteq \ti$, $\tro^*\ti\tro\subseteq\es$, $[\tro^*\tro]^\closure$ and $[\tro\tro^*]^\closure$ are unital and therefore we get:
    \[
    \es=[\tro^*\tro]^\closure\es[\tro^*\tro]^\closure\subseteq [\tro^*\tro\es\tro^*\tro]^\closure\subseteq [\tro^*\ti\tro]^\closure\subseteq \es
    \]which establishes that $\es=[\tro^*\ti\tro]^\closure$ and $\ti=[\tro\es\tro^*]^\closure$ by the same arguments. 
\end{remark}
The proof of the following Proposition follows the same arguments as in the proof of \cite[Theorem 2.1]{ele} so we omit it:

\begin{theorem}
\label{troeq}
    TRO equivalence is an equivalence relation.
\end{theorem}
Based on the previous, we define $\Delta$-equivalence for (abstract) dual operator systems:
\begin{definition}
    Two (abstract) dual operator systems $\es$, $\ti$ are called $\Delta$-equivalent in case there exist normal unital complete order embeddings $\phi:\es\to B(H)$, $\psi:\ti\to B(K)$ such that $\phi(\es)\sim_{TRO}\psi(\ti)$.
\end{definition}

Unlike TRO equivalence, the fact that $\Delta$-equivalence is an equivalence relation requires a bit of work. We will follow the main ideas as in the proof of the analogous result in \cite{strspaces}, with minor variations, so that they adapt to the dual case. We start by proving the following:

\begin{lemma}
\label{trom}
    Let $\es\subseteq B(H)$ and $\ti\subseteq B(K)$ be $w^*$-closed operator subsystems with $\es\sim_\tro\ti$ and suppose $\phi:\es\to B(H_0)$ is a normal unital completely positive map such that $\phi$ restricted to $[\tro^*\tro]^\closure$ is a $*$-homomorphism. Then, there exist a Hilbert space $K_0$ and a normal unital completely positive map $\psi:\ti\to B(K_0)$, such that:
    \begin{itemize}
        \item $\psi$ restricted to $[\tro\tro^*]^\closure$ is a $*$-homomorphism, and
        \item $\overline{\phi(\es)}^{w^*}\sim_{TRO}\overline{\psi(\ti)}^{w^*}$.
    \end{itemize}
    If in addition $\phi$ is a complete order embedding (equivalently a complete isometry), $\psi$ can be chosen to be a complete order embedding too, therefore their ranges are $w^*$-closed and $\phi(\es)\sim_{TRO}\psi(\ti)$.
\end{lemma}

\begin{proof}
    We may assume that $\tro$ is $w^*$-closed (otherwise we can replace it by the TRO $\ntro=\overline{\tro}^{w^*}$ for which $\es\sim_\ntro\ti$, $[\tro^*\tro]^\closure=[\ntro^*\ntro]^\closure$ and $[\tro\tro]^\closure=[\ntro\ntro]^\closure$). First, we note that by \cite[1.3.12]{blecher} (combined with Arveson's extension Theorem), we have $\phi(ax)=\phi(a)\phi(x)$ and thus also $\phi(xa)=\phi(x)\phi(a)$ for all $a\in[\tro^*\tro]^\closure$, $x\in\es$. By \cite[Lemma 8.5.23]{blecher}, there exists a vector $\en=(n_i)_{i \in J}$ in  $R_J^w(\tro)$, consisting of partial isometries $n_i\in\tro$, such that $\sum_{i\in J}n_in_i^*$ converges to $I_K$ in the $w^*$-topology, i.e. $\en\en^*=I_K$. With $H_0$ considered as a left $[\tro^*\tro]^\closure$-module (by the action of $\phi)$, we set $K_0:=\tro\otimes_\phi H_0$ be the completion of the algebraic $[\tro^*\tro]^\closure$-balanced tensor product of $\tro$ with $H_0$ equipped with inner product
    \[
    \langle m_1\otimes\xi_1,m_2\otimes\xi_2\rangle:=\langle\phi(m_2^*m_1)\xi_1,\xi_2\rangle.
    \]As in \cite{strspaces}, we define the TRO morphism $t:\tro\to B(H_0,K_0)$ by $t(m)\xi:=m\otimes\xi$, which satisfies $t(m_1)^*t(m_2)=\phi(m_1^*m_2)$ for all $m_1,m_2\in\tro$. 

    We define $\psi:\ti\to B(K_0)$ as
    \[
    \psi(y):=t^{(1,J)}(\en)\phi^{(J)}(\en^*y\en)t^{(1,J)}(\en)^*=\sum_{i,j}t(n_i)\phi(n_i^*yn_j)t(n_j)^*
    \]It is easily seen that $\psi$ is a normal unital completely contractive (hence completely positive) map and that $\psi(m_1m_2^*)=t(m_1)t(m_2)^*$ for all $m_1,m_2\in\tro$, from which it follows that $\psi$ restricted to $[\tro\tro^*]^\closure$ is a $*$-homomorphism. Furthermore,
    \[
    \psi(axb^*)=t(a)\phi(x)t(b)^* \ \ \text{and}
    \]
    \[\phi(a^*yb)=t(a)^*\psi(y)t(b)
    \]for all $a,b\in\tro$, $x\in\es$, $y\in\ti$ and this, combined with Remark \ref{nondeg} and the fact that $t(\tro)$ acts nondegenerately, implies that $\overline{\phi(\es)}^{w^*}\sim_{TRO}\overline{\psi(\ti)}^{w^*}$ via $t(\tro)$. 

    In case $\phi$ is completely isometric, for all $y\in\ti$ we have
    \begin{gather*}
        \|\psi(y)\|=\|(t^{(1,J)}(\en)^*t^{(1,J)}(\en))^{1/2}\phi^{(J)}(\en^*y\en)(t^{(1,J)}(\en)^*t^{(1,J)}(\en))^{1/2}\| \\ =\|\phi^{(J)}(\en^*\en)^{1/2}\phi^{(J)}(\en^*y\en)\phi^{(J)}(\en^*\en)^{1/2}\|=\|(\en^*\en)^{1/2}\en^*y\en(\en^*\en)^{1/2}\|=\|y\| 
    \end{gather*}thus $\psi$ is isometric, and an identical argument proves that $\psi$ is completely isometric.
\end{proof}

We state the following Lemma, which follows essentially from the definition of the left adjointable multiplier algebra:

\begin{lemma}
\label{useful}
    Let $\es$ be an operator system and $\phi:\es\to B(H)$ a unital complete order embedding. Suppose $T\in B(H)$ satisfies $T\phi(\es)\subseteq\phi(\es)$ and $T^*\phi(\es)\subseteq\phi(\es)$. Then, $T\in \phi(\alfa_\es)$. Furthermore, the map $\phi^{-1}$ from the $C^*$-algebra ${\mathcal{A}}:=\{T\in B(H):T\phi(\es)\subseteq\phi(\es),T^*\phi(\es)\subseteq\phi(\es)\}$ into $\alfa_\es$ is a $*$-homomorphism (here $\alfa_\es$ is equipped with the product it inherits from $A_\ell(\es)$).
\end{lemma}

\begin{proof}
    We define the map $\alpha_T:\es\to\es$, $\alpha_T(s):=\phi^{-1}(T\phi(s))$. Clearly, $\alpha_T\in A_\ell(\es)$. Hence, by definition of $\alfa_\es$, $\alpha_T(1)=\phi^{-1}(T)\in \alfa_\es$ or $T\in \phi(\alfa_\es)$. The second statement follows from the fact that the map $T\mapsto \alpha_T$ is a $*$-homomorphism.
\end{proof}

\begin{theorem}
\label{delta}
    $\Delta$-equivalence is an equivalence relation.
\end{theorem}

\begin{proof}
    Reflexivity and symmetricity are obvious. For transitivity, let $\es_1\sim_\Delta\es_2$ and $\es_2\sim_\Delta \es_3$. There exist normal unital complete order embeddings $\phi:\es_1\to B(K)$, $\psi_1:\es_2\to B(H_1)$, $\psi_2:\es_2\to B(H_2)$, $\omega:\es_3\to B(L)$ such that $\phi(\es_1)\sim_{TRO}\psi_1(\es_2)$ and $\psi_2(\es_2)\sim_{TRO}\omega(\es_3)$. Let $\gamma:\es_2\to B(H_0)$ be a normal representation as in Remark \ref{rep2} (i.e. $\gamma$ is a $*$-homomorphism when restricted to $\alfa_{\es_2}$). Suppose $\phi(\es_1)\sim_\tro\psi_1(\es_2)$. Then, $\psi_1(\es_2)$ is a $[\tro\tro^*]^\closure$-bimodule, so by Lemma \ref{useful}, $[\tro\tro^*]^\closure\subseteq \psi_1(\alfa_{\es_2})$ and the restriction $\psi_1^{-1}:[\tro\tro^*]^\closure\to \alfa_{\es_2}$ is a $*$-homomorphism, thus $\gamma\circ \psi_1^{-1}$ restricted to $[\tro\tro^*]^\closure$ is a $*$-homomorphism too. By Lemma \ref{trom}, there exist a Hilbert space $K_0$ and a normal unital complete order embedding $\alpha:\phi(\es_1)\to B(K_0)$ such that
    \[
    \alpha(\phi(\es_1))\sim_{TRO}\gamma\circ\psi_1^{-1}(\psi_1(\es_2))=\gamma(\es_2).
    \]
    
    Using the same argument applied to the equivalence $\psi_2(\es_2)\sim_{TRO}\omega(\es_3)$, we get a normal unital complete order embedding $\beta:\omega(\es_3)\to B(L_0)$ such that
    \[
    \gamma(\es_2)\sim_{TRO}\beta(\omega(\es_3))
    \]thus, by Theorem \ref{troeq}, $\alpha(\phi(\es_1))\sim_{TRO}\beta(\omega(\es_3))$, or $\es_1\sim_\Delta\es_3$.
\end{proof}
\begin{remark}
\label{bimap}
    The proof of Theorem \ref{delta} yields a result, which will be useful later: if $\es\sim_\Delta\ti$, there exist normal unital complete order embeddings $\Phi:\es\to B(H)$ and $\psi:\ti\to B(K)$ so that $\Phi|_{\alfa_\es}$ is a $*$-homomorphism and $\Phi(\es)\sim_{TRO}\psi(\ti)$.
\end{remark}

\section{Stable isomorphism} \label{sec:weakstable}

\subsection{Dual operator spaces}

We recall the following from \cite{stable}:
\begin{definition}
    \begin{itemize}
        \item[(i)] Two $w^*$-closed subspaces $\spacex\subseteq B(H_1,H_2),\spacey\subseteq B(K_1,K_2)$ are called \textit{TRO-equivalent} if there exist TRO's $\tro_1\subseteq B(H_1,K_1)$ and $\tro_2\subseteq B(H_2,K_2)$ with $\spacey=[\tro_2\spacex\tro_1^*]^\closure$ and $\spacex=[\tro_2^*\spacey\tro_1]^\closure$. As before, we write $\spacex\sim_{TRO}\spacey$.
        \item[(ii)] Two (abstract) dual operator spaces $\spacex,\spacey$ are called \textit{$\Delta$-equivalent} if there exist normal completely isometric representations $\phi:\spacex\to B(H_1,H_2)$ and $\psi:\spacey\to B(K_1,K_2)$ such that $\phi(\spacex)\sim_{TRO}\psi(\spacey)$. In that case, we write $\spacex\sim_\Delta\spacey$.
        \item[(iii)] Two dual operator spaces $\spacex,\spacey$ are called \textit{stably isomorphic} if there exists a cardinal $J$ such that $\spacex\nsp M_J\cong \spacey\nsp M_J$ completely isometrically and $w^*$-homeomorphically. 
    \end{itemize}
\end{definition}

\begin{theorem}{\cite[Theorem 2.5]{stable}}
\label{paulseneq}
    Two dual operator spaces are $\Delta$-equivalent if and only if they are stably isomorphic.
\end{theorem}
The "if" part follows directly from the fact that $\Delta$-equivalence is an equivalence relation, coarser than normal completely isometric isomorphism, and $A\sim_\Delta A\nsp M_J$ for all cardinals $J$ and all operator spaces $A$, which is easy to see.

The "only if" part requires a fair amount of work, and the proof relies on the analogous result about dual operator algebras \cite[Theorem 3.2]{stalgebras}. As we shall see, Theorem \ref{paulseneq} follows from a much stronger result (Theorem \ref{uni} below). 

If $\spacex\sim_{TRO}\spacey$ via $w^*$-closed TRO's $\tro_1,\tro_2$ as above, and $\spacex,\spacey$ act nondegenerately, then $\tro_1,\tro_2$ act nondegenerately too, therefore by \cite[Lemma 8.5.23]{blecher} there exists a cardinal $J$ (which we can assume to be $\geq |\mathbb{N}|$) and vectors $\emm_1=(m_{1,i})_{i\in J}\in \tro_1\nsp C_J$, $\emm_2=(m_{2,i})_{i\in J}\in\tro_2\nsp C_J$, $\en_1=(n_{1,i})_{i\in J}\in \tro_1\nsp R_J$ and $\en_2=(n_{2,i})_{i\in J}\in\tro_2\nsp R_J$ such that $\emm_i^*\emm_i=I_{H_i}$ and $\en_i\en_i^*=I_{K_i}$ for $i=1,2$. We define $$p_i:=\emm_i\emm_i^*\in [\tro_i\tro_i^*]^\closure\nsp M_J, \ i=1,2$$ which are projections in $B(K_i\otimes \ell_J^2)$ and $$q_i:=\en_i^*\en_i\in[\tro_i^*\tro_i]^\closure\nsp M_J, \ i=1,2$$ which are projections in $B(H_i\otimes \ell_J^2)$. 

Using a well-known trick (see e.g. the proofs of \cite[Theorem 8.5.28]{blecher}, \cite[Lemma 3.5]{ele}, \cite[Theorem 4.6]{strspaces} and \cite[Lemma 3.1]{stalgebras}), one can show that $\spacex \nsp M_J\cong \spacey \nsp M_J$ completely isometrically. As it turns out, this complete isometry is given by unitary operators:
\begin{theorem}
\label{uni}
    Suppose that $\spacex\subseteq B(H_1,H_2)$, $\spacey\subseteq B(K_1,K_2)$ are $w^*$-closed subspaces acting nondegenerately with $\spacex\sim_{TRO}\spacey$. Then, there exists a cardinal $J$ and unitary operators $U_i\in B(H_i\otimes \ell_J^2,K_i\otimes \ell_J^2)$, $i=1,2$ such that
    \[
    \spacey\nsp M_J=U_2(\spacex \nsp M_J)U_1^*.
    \]In particular, $X$ and $Y$ are stably isomorphic.
\end{theorem}

\begin{proof}
    Let $J, \tro_1,\tro_2,\emm_1,\emm_2,\en_1,\en_2$ be as above. Since $|J|\geq |\mathbb{N}|$, there is a unitary $V\in B(\ell_J^2,\ell_J^2\otimes\ell^2\otimes\ell_J^2)$.

    For $i=1,2$, we define $U_i:=S_i+T_i$, where
    \begin{align} \label{ff}
    S_i:=(I_{K_i}\otimes V^*)(\emm_i\otimes e_1\otimes V^*) ((q_i\otimes W_L+q_i^\perp\otimes I_{\ell^2})\otimes I_{\ell_J^2})(I_{H_i}\otimes V)
    \\ \label{ff2}
    T_i:=(I_{K_i}\otimes V^*)((p_i\otimes W_R+p_i^\perp\otimes I_{\ell^2})\otimes I_{\ell_J^2})(\en_i\otimes e_1^*\otimes V)(I_{H_i}\otimes V)
    \end{align}
    where $W_L,W_R\in B(\ell^2)$ are the left- and right-shift operators respectively. Direct computations show that
    \begin{itemize}
        \item $S_iS^*_i+T_iT_i^*=I_{K_i}\otimes I_{\ell_J^2}$, $S_i^*S_i+T_i^*T_i=I_{H_i}\otimes I_{\ell_J^2}$ and
        \item $S_iT_i^*$, $S_i^*T_i$, $T_iS_i^*$, $T_i^*S_i$ are all zero,
    \end{itemize}
    therefore $U_1,U_2$ are unitary operators and the fact $\spacey \nsp M_J=U_2(\spacex \nsp M_J)U_1^*$ is easily seen.
\end{proof}

We obtain a bunch of immediate corollaries:

\begin{corollary} \label{corollary:algebras}
    Suppose that $A\subseteq B(H)$, $B\subseteq B(K)$ are unital $w^*$-closed subalgebras which are TRO-equivalent, in the sense that there exists a TRO $\tro\subseteq B(H,K)$ with $A=[\tro^*B\tro]^\closure$ and $B=[\tro A\tro^*]^\closure$. Then there exists a cardinal $J$ and a unitary $U\in B(H\otimes\ell^2,K\otimes\ell^2)$ such that $B\nsp M_J=U(A\nsp M_J)U^*$.
\end{corollary}

\begin{proof}
    We apply the proof of Theorem \ref{uni} with $U=U_1=U_2$.
\end{proof}

\begin{corollary}
    Suppose $\mathcal{A}, B$ are von Neumann algebras acting on $H,K$ respectively and that $\mathcal{A}',B'$ are $*$-isomorphic. Then, for some cardinal $J$, ${\mathcal{A}}\nsp M_J$ and $\mathcal{B}\nsp M_J$ are unitarily equivalent.  
\end{corollary}

\begin{proof}
    By \cite[Theorem 3.2]{troeq} (see also \cite[Corollary 8.5.38]{blecher}) $\mathcal{A},B$ are TRO-equivalent and then Corollary \ref{corollary:algebras} applies.
\end{proof}
Using Corollary \ref{corollary:algebras} combined with \cite[Theorem 3.3]{troeq}, we get the following:
\begin{corollary}
    Suppose $A\subseteq B(H)$, $B\subseteq B(K)$ are unital reflexive algebras and that there exists a $*$-isomorphism $\theta:\Delta(A)'\to \Delta(B)'$ such that $\theta(\lat A)=\lat B$. Then, $A\nsp M_J$ and $B\nsp M_J$ are unitarily equivalent for some cardinal $J$.
\end{corollary}

As a final result of this section, the following Proposition is a stronger version of \cite[Theorem 2.14]{stable}:

\begin{proposition}
\label{CEStheorem}
    Let $\spacex,\spacey$ be $\Delta$-equivalent dual operator spaces and $(\pi,\phi,\sigma)$ a normal CES representation of the dual operator $A_\ell(\spacex)-A_r(\spacex)$-bimodule $\spacex$. Then, there exists a normal CES representation $(\rho,\psi,\tau)$ of the dual operator $A_\ell(\spacey)-A_r(\spacey)$-bimodule $\spacey$ such that $\phi(\spacex)$ is TRO-equivalent to $\psi(\spacey)$. Furthermore there exist TRO's $\tro_1$, $\tro_2$ implementing the TRO equivalence, so that 
    \begin{equation}
        \label{properties}
        \begin{split} 
        \pi(A_\ell(\spacex))=[\tro_2^*\tro_2]^\closure, \sigma(A_r(\spacex))=[\tro_1^*\tro_1]^\closure, \\ \rho(A_\ell(\spacey))=[\tro_2\tro_2^*]^\closure, \tau(A_r(\spacey))=[\tro_1\tro_1^*]^\closure.
    \end{split}
    \end{equation}
    
\end{proposition}

\begin{proof}
    We can assume without loss of generality that $\phi$ is nondegenerate, for if the result holds for nondegenerate CES representations, given any normal CES representation $(\pi,\phi,\sigma)$ (with $\phi:\spacex\to B(H_1,H_2)$), by setting $\tilde{H}_1:=[\phi(\spacex)^*H_2]$ and $\tilde{H}_2:=[\phi(\spacex)H_1]$, the compressions $\tilde{\pi}:=P_{\tilde{H}_2}\pi(\cdot)|_{\tilde{H}_2}$, $\tilde{\phi}:=P_{\tilde{H}_2}\phi(\cdot)|_{\tilde{H}_1}$ and $\tilde{\sigma}:=P_{\tilde{H}_1}\sigma(\cdot)|_{\tilde{H}_1}$ form a normal nondegenerate CES representation $(\tilde{\pi},\tilde{\phi},\tilde{\sigma})$. Then, if $(\rho,\psi,\tau)$ (with $\psi:\spacey\to B(K_1,K_2)$) is a normal CES representation and $\tro_1\subseteq B(\tilde{H}_1,K_1)$, $\tro_2\subseteq B(\tilde{H}_2,K_2)$ are TRO's as in the conclusion of the Theorem, the TRO's $\mathcal{N}_1:=\tro_1P_{\tilde{H}_1}\subseteq B(H_1,K_1)$ and $\mathcal{N}_2:=\tro_2P_{\tilde{H}_2}\subseteq B(H_2,K_2)$ implement a TRO equivalence between $\phi(\spacex)$ and $\psi(\spacey)$ with the required properties.
    
    By \cite[Theorem 2.14]{stable}, given $(\pi,\phi,\sigma)$ (with $\phi:\spacex\to B(H_1,H_2)$), there exists a normal complete isometry $\psi:\spacey\to B(K_1,K_2)$ such that $\phi(\spacex)\sim_{TRO}\psi(\spacey)$ and we may assume that $\phi$ is nondegenerate. By \cite[Proposition 2.2]{stable}, we may also assume that $\psi$ is nondegenerate. We have that
    \[
    \pi(A_\ell(\spacex))=\Delta(\text{alg}_\ell(\phi(\spacex))).
    \]One inclusion is obvious, and if $T$ is in the right hand side, the map $\alpha_T:\spacex\to\spacex$, $\alpha_T(x):=\phi^{-1}(T\phi(x))$ is by definition in $A_\ell(\spacex)$, and nondegeneracy of $\phi$ implies that $T=\pi(\alpha_T)$. We define a von Neumann algebra $\mathcal{A}\subseteq B(K_2)$ as:
    \[
    \mathcal{A}:=\Delta(\text{alg}_\ell(\psi(\spacey)))
    \]and we get a normal unital $*$-homomorphism 
    \[
    \omega:\mathcal{A}\to A_\ell(\spacey)
    \]
    \[
    \omega(T)(y):=\psi^{-1}(T\psi(y))
    \]which is also faithful, due to the fact that $\psi$ is nondegenerate. We shall prove that $\omega$ is surjective too. By Theorem \ref{uni}, there exist a cardinal $J$ and unitary operators $U_1,U_2$ such that
    \begin{equation} \label{uniteq}
    \psi(\spacey)\nsp M_J=U_2(\phi(\spacex)\nsp M_J)U_1^*.
    \end{equation}
    One might easily observe that
    \[
    \pi(A_\ell(\spacex))\nsp M_J=\Delta(\text{alg}_\ell(\phi(\spacex)\nsp M_J))
    \]
    \[
    \mathcal{A}\nsp M_J=\Delta(\text{alg}_\ell(\psi(\spacey)\nsp M_J))
    \]and by \eqref{uniteq}, a direct computation yields 
    \[
    \Delta(\text{alg}_\ell(\psi(\spacey)\nsp M_J))=U_2\Delta(\text{alg}_\ell(\phi(\spacex)\nsp M_J))U_2^*
    \]therefore we get
    \begin{equation} \label{unitadjoint}
        \mathcal{A}\nsp M_J=U_2(\pi(A_\ell(\spacex))\nsp M_J)U_2^*.
    \end{equation}
    It follows that $\mathcal{A}\nsp M_J\cong A_\ell(\spacex)\nsp M_J$
    and by \eqref{uniteq} we have
    \[
    A_\ell(\spacex\nsp M_J)\cong A_\ell(\spacey\nsp M_J)
    \]so that $A_\ell(\spacex)\nsp M_J\cong A_\ell(\spacey)\nsp M_J$ using \cite[Theorem 5.10.1]{adjoint}. After computations, one can show that the resulting isomorphism $\mathcal{A}\nsp M_J\cong A_\ell(\spacey)\nsp M_J$ is in fact $\omega\otimes I$, which proves that $\omega$ is surjective. Setting $$\rho:=\omega^{-1}:A_\ell(\spacey)\to \mathcal{A}\subseteq B(K_2)$$ by definition of $\omega$ we have $\psi(a\cdot y)=\rho(a)\psi(y)$ for all $a\in A_\ell(\spacey)$, $y\in\spacey$. 

    Similarly, one can find a map $\tau:A_r(\spacey)\to B(H_2)$ so that $(\rho,\psi,\tau)$ is a normal CES representation.

    Assume now that $\ntro_1,\ntro_2$ implement the TRO equivalence of $\phi(\spacex),\psi(\spacey)$, i.e.
    \[
    \phi(\spacex)=[\ntro_2^*\psi(\spacey)\ntro_1]^\closure \ \ \text{and} \ \ \psi(\spacey)=[\ntro_2\phi(\spacex)\ntro_1^*]^\closure. 
    \]Clearly we have
    \[
    \ntro_2\pi(A_\ell(\spacex))\ntro_2^*\psi(\spacey)\subseteq\psi(\spacey), \ \ntro_2^*\rho(A_\ell(\spacey))\ntro_2\phi(\spacex)\subseteq\phi(\spacex)
    \]therefore $\ntro_2\pi(A_\ell(\spacex))\ntro_2^*\subseteq \rho(A_\ell(\spacey))$, $\ntro_2^*\rho(A_\ell(\spacey))\ntro_2\subseteq \pi(A_\ell(\spacex))$, and since $\ntro_1$ and $\ntro_2$ act nondegenerately,
    \[
    \pi(A_\ell(\spacex))=[\ntro_2^*\rho(A_\ell(\spacey))\ntro_2]^\closure, \ \rho(A_\ell(\spacey))=[\ntro_2\pi(A_\ell(\spacex))\ntro_2^*]^\closure
    \]and by similar arguments
    \[
    \sigma(A_r(\spacex))=[\ntro_1^*\tau(A_r(\spacey))\ntro_1]^\closure, \ \tau(A_r(\spacey))=[\ntro_1\sigma(A_r(\spacex))\ntro_1^*]^\closure.
    \]Finally, the TRO's $$\tro_1:=[\tau(A_r(\spacey))\ntro_1\sigma(A_r(\spacex))]^\closure, \ \tro_2:=[\rho(A_\ell(\spacey))\ntro_2\pi(A_\ell(\spacex))]^\closure$$ implement a TRO equivalence between $\phi(\spacex)$ and $\psi(\spacey)$ satisfying \eqref{properties}.
\end{proof}

\subsection{Dual operator systems}

In the category of dual operator systems, the morphisms are normal unital completely positive maps (equivalently, normal unital complete contractions), thus we call two dual operator systems $\es,\ti$ \textit{stably isomorphic} in case there exists a cardinal $J$ such that $\es\nsp M_J$ and $\ti\nsp M_J$ are isomorphic as dual operator systems, i.e. when there exists a normal unital complete order isomorphism from $\es\nsp M_J$ onto $\ti\nsp M_J$. The following is an analogue of Theorem \ref{uni}:

\begin{theorem}
\label{unisys}
    Let $\es\subseteq B(H)$, $\ti\subseteq B(K)$ be $w^*$-closed operator systems with $\es\sim_{TRO}\ti$ as dual operator systems. Then, there exists a cardinal $J$ and a unitary operator $U\in B(H\otimes \ell_J^2,K\otimes\ell_J^2)$ such that 
    \[
    \ti\nsp M_J=U(\es\nsp M_J)U^*.
    \]
\end{theorem}
\begin{proof}
    By definition, there exists a TRO $\tro\subseteq B(H,K)$ with $\ti=[\tro\es\tro^*]^\closure$ and $\es=[\tro^*\ti\tro]^\closure$. Then, Theorem \ref{uni} applies with $U=U_1=U_2$.
\end{proof} 

\begin{corollary} \label{dualopsysstable}  
    Two dual operator systems are $\Delta$-equivalent if and only if they are stably isomorphic.
\end{corollary}

An analogue of Proposition \ref{CEStheorem} holds for dual operator systems too. It is a dual version of \cite[Proposition 3.9]{systems}, but with a completely different proof, as the proof of the latter relies on the theory of $C^*$-envelopes, which do not appear in a dual context in general:

\begin{proposition} \label{CESsystems}
    Suppose $\es\sim_\Delta\ti$. Then, there exist normal unital complete order embeddings $\phi:\es\to B(H)$, $\psi:\ti\to B(K)$ and a TRO $\tro\subseteq B(H,K)$ such that the restrictions of $\phi$ and $\psi$ to $\alfa_\es$ and $\alfa_\ti$ respectively are $*$-homomorphisms, and $\phi(\es)\sim_\tro\psi(\ti)$. Furthermore, $\phi(\alfa_\es)=[\tro^*\tro]^\closure$, $\psi(\alfa_\ti)=[\tro\tro^*]^\closure$, thus $\phi(\alfa_\es)\sim_\tro\psi(\alfa_\ti)$. 
\end{proposition}

\begin{proof}
    By Remark \ref{bimap}, we can pick normal unital complete order embeddings $\phi:\es\to B(H)$, $\psi:\ti\to B(K)$ and a TRO $\ntro\subseteq B(H,K)$ so that $\phi(\es)\sim_\ntro \psi(\ti)$ and the restriction of $\phi$ to $\alfa_\es$ is a $*$-homomorphism. Note that with $i_\es:A_\ell(\es)\to \es$ and $j_\es:A_r(\es)\to \es$ the completely isometric embeddings $u\mapsto u(1)$, $(\phi\circ i_\es,\phi,\phi\circ j_\es)$ is a normal faithful nondegenerate CES representation of $\es$ considered as a $A_\ell(\es)-A_r(\es)$-bimodule. By (the proof of) Proposition \ref{CEStheorem}, there exist normal unital $*$-homomorphisms $\rho:A_\ell(\ti)\to B(K)$ and $\tau:A_r(\ti)\to B(K)$ such that $(\rho,\psi,\tau)$ is a normal faithful CES representation of the $A_\ell(\ti)-A_r(\ti)$-bimodule $\ti$ and 
    \begin{equation}
    \begin{split}
        \phi(\alfa_\es)=\phi\circ i_\es(A_\ell(\es))\sim_\ntro \rho(A_\ell(\ti)), \\ \phi(\alfa_\es)=\phi\circ j_\es(A_r(\es))\sim_\ntro \tau(A_r(\ti)).
        \end{split}
    \end{equation}
    For $i_\ti:A_\ell(\ti)\to \ti$ and $j_\ti:A_r(\ti)\to \ti$ the embeddings $u\mapsto u(1)$, we clearly have $\rho=\psi\circ i_\ti$ and $\tau=\psi\circ j_\ti$, therefore $\psi|_{\alfa_\ti}$ is a $*$-homomorphism and $\phi(\alfa_\es)\sim_\ntro \psi(\alfa_\ti)$. Finally, setting $\tro:=[\psi(\alfa_\ti)\ntro\phi(\alfa_\es)]^\closure$ completes the proof.
\end{proof}

\section{The strong case} \label{sec:strongstable}

\subsection{Stably isomorphic operator spaces}

Recall \cite{strspaces} that two norm-closed subspaces $\spacex\subseteq B(H_1,H_2),\spacey\subseteq B(K_1,K_2)$ are called strongly TRO-equivalent if there exist TRO's $\tro_1\subseteq B(H_1,K_1)$ and $\tro_2\subseteq B(H_2,K_2)$ such that
\[
\spacey=[\tro_2\spacex\tro_1^*]^\normcl \ \ \text{and} \ \ \spacex=[\tro_2^*\spacey\tro_1]^\normcl.
\]

Two (abstract) operator spaces $\spacex,\spacey$ are called strongly $\Delta$-equivalent if there exist complete isometries $\phi:\spacex\to B(H_1,H_2)$, $\psi:\spacey\to B(K_1,K_2)$ such that $\phi(\es)\sim_{TRO}\psi(\ti)$. They are called \textit{stably isomorphic} in case $\spacex\otimes \kk$ and $\spacey\otimes \kk$ are completely isometrically isomorphic. By \cite[Theorem 4.3]{strspaces}, if $\spacex,\spacey$ are stably isomorphic then they are $\Delta$-equivalent. As indicated by the case of $C^*$- algebras, where there exist examples of strongly Morita equivalent $C^*$-algebras which are not stably isomorphic \cite{moritacstar}, the converse is not expected to be true in general.

However, the converse statement is also true if one replaces $\Delta$-equivalence by $\sigma\Delta$-equivalence. We recall the following definitions from \cite{stablemaps}:

\begin{definition} \label{sigmatro}
    Let $\tro\subseteq B(H,K)$ be a norm-closed TRO. $\tro$ is called a \textit{$\sigma$-TRO} in case there exist sequences $(m_k)_{k\in\mathbb{N}}$ and $(n_k)_{k\in\mathbb{N}}$ in $\tro$ such that
    \[
    \left\|\sum_{i=1}^k m_i^*m_i\right\|\leq 1  \ \ \text{and} \ \  \left\| \sum_{i=1}^k n_in_i^*\right\|\leq 1 \ \ \text{for all $k$,}
    \]that satisfy 
    \[
    \sum_{k=1}^\infty mm_k^*m_k=m, \ \sum_{k=1}^\infty n_kn_k^*n=n
    \]as norm-convergent sums for all $m,n\in\tro$.
\end{definition}

Two norm-closed subspaces $\spacex\subseteq B(H_1,H_2)$, $\spacey\subseteq B(K_1,K_2)$ are called \textit{$\sigma$-strongly TRO-equivalent} in case there exist $\sigma$-TRO's $\tro_i\subseteq B(H_i,K_i)$, $i=1,2$ such that
\[
\spacex=[\tro_2^*\spacey\tro_1]^\normcl , \ \spacey=[\tro_2\spacex\tro_1^*]^\normcl.
\]Two (abstract) operator spaces $\spacex,\spacey$ are called \textit{$\sigma\Delta$-equivalent} in case there exist completely isometric representations $\phi:\spacex\to B(H_1,H_2)$, $\psi:\spacey\to B(K_1,K_2)$ such that $\phi(\spacex)$ and $\psi(\spacey)$ are $\sigma$-strongly TRO-equivalent. We have the following:

\begin{theorem}{\cite[Theorem 4.6]{strspaces}} \label{thmelekak}
    If $\spacex\subseteq B(H_1,H_2)$ and $\spacey\subseteq B(K_1,K_2)$ are $\sigma$-strongly TRO-equivalent, then $\spacex$ and $\spacey$ are stably isomorphic.
\end{theorem}
Following the same steps of the proof of Theorem \ref{uni}, we show something stronger:

\begin{theorem} \label{opspc}
    If $\spacex\subseteq B(H_1,H_2)$ and $\spacey\subseteq B(K_1,K_2)$ are $\sigma$-strongly TRO-equivalent and they act nondegenerately, then there exist unitary operators $U_i\in B(H_i\otimes\ell^2,K_i\otimes\ell^2)$ such that 
    \[
    \spacey\otimes\kk=U_2(\spacex\otimes\kk) U_1^*.
    \]
\end{theorem}

\begin{proof}
    Let $\tro_i\subseteq B(H_i,K_i)$, $i=1,2$, be the $\sigma$-TRO's implementing the equivalence and $(m_{i,k})_{k\in\mathbb{N}}$, $(n_{i,k})_{k\in\mathbb{N}}$ be sequences in $\tro_i$ as in Definition \ref{sigmatro}. 
    We set $\emm_i:=(m_{i,k})_{k=1}^\infty$, $\en_i:=(n_{i,k})_{k=1}^\infty$, considered as operators in $B(H_i,K_i\otimes\ell^2)$ and $B(H_i\otimes\ell^2,K_i)$ respectively, for $i=1,2$. The fact that $\spacex,\spacey$ act nondegenerately, implies that $\emm_i^*\emm_i=I_{H_i}$ and $\en_i\en_i^*=I_{K_i}$ for $i=1,2$. Setting
    \[
    p_i:=\emm_i\emm_i^*\in B(K_i\otimes \ell^2)
    \]
    \[
    q_i:=\en_i^*\en_i\in  B(H_i\otimes\ell^2)
    \]and $U_i:=S_i+T_i$ with the same notation as in the proof of Theorem \ref{uni}, with $J=\mathbb{N}$, where $S_i,T_i$ are given by \eqref{ff} and \eqref{ff2} respectively, we have that $U_1,U_2$ are unitary operators. It remains to show that $\spacey\otimes\kk=U_2(\spacex\otimes\kk)U_1^*$. 

    We argue why $S_2(\spacex\otimes\kk)\subseteq[\tro_2\spacex]^\normcl\otimes\kk$ and by analogous arguments $T_2(\spacex\otimes\kk)\subseteq [\tro_2\spacex]^\normcl\otimes\kk$ and $([\tro_2\spacex]^\normcl\otimes\kk)U_1^*\subseteq [\tro_2\spacex\tro_1^*]^\normcl\otimes\kk=\spacey\otimes\kk$. Then, it follows that $U_2(\spacex\otimes\kk)U_1^*\subseteq \spacey\otimes\kk$ and the fact $U_2^*(\spacey\otimes\kk)U_1\subseteq \spacex\otimes\kk$ is proved identically.

    To this end, the only nontrivial part is to show $q_2(\spacex\otimes\kk)\subseteq \spacex\otimes\kk$ and $\emm_2\spacex\subseteq [\tro_2\spacex]^\normcl\otimes C_\infty$. The first inclusion follows from the fact that
    \[
    q_2(x\otimes E_{k\ell})=\sum_{i=1}^\infty n_{2,i}^*n_{2,k}x\otimes E_{i\ell}
    \]as a norm convergent sum for all $k,\ell\in\mathbb{N}$ and $x\in\spacex$. An analogous argument proves that 
    \[
    \emm_2x=\sum_{i=1}^\infty m_{2,i}x\otimes e_i
    \]as a norm-convergent sum for all $x\in\spacex$, which proves the second inclusion and therefore completes the proof.
\end{proof}

\begin{remark}
    We note that the map $X\mapsto U_2XU_1^*$ in the above proof is the same map as the one constructed in the proof of \cite[Theorem 4.6]{strspaces}, although it is not mentioned in \cite{strspaces} that the map comes from a unitary equivalence.
\end{remark}

\subsection{Stably isomorphic operator systems}

Recall \cite{systems} that two concrete operator systems $\es\subseteq B(H),\ti\subseteq B(K)$ are called strongly TRO-equivalent if there exists a TRO $\tro\subseteq B(H,K)$ such that $I_H\in[\tro^*\tro]^\normcl$, $I_K\in [\tro\tro^*]^\normcl$ (i.e. $\tro$ is non-degenerate) and
\[
\tro\es\tro^*\subseteq \ti \ \ \text{and} \ \ \tro^*\ti\tro\subseteq \es.
\]Using \cite[Lemma 4.9]{strspaces}, an equivalent characterization is that there exists a norm-closed TRO such that
\[
\ti=[\tro\es\tro^*]^\normcl \ \ \text{and} \ \ \es=[\tro^*\ti\tro]^\normcl.
\]

Two (abstract) operator systems $\es,\ti$ are called strongly \textit{$\Delta$-equivalent} if there exist unital complete order embeddings $\phi:\es\to B(H)$, $\psi:\ti\to B(K)$ such that $\phi(\es)\sim_{TRO}\psi(\ti)$.

We have the following:
\begin{theorem} \label{strsys}
    Let $\es\subseteq B(H)$, $\ti\subseteq B(K)$ be strongly TRO-equivalent operator systems. Then, there exists a unitary operator $U\in B(H\otimes\ell^2,K\otimes\ell^2)$ such that 
    \[
    \ti\otimes\kk=U(\es\otimes\kk)U^*.
    \]
\end{theorem}

\begin{proof}
    Suppose $\tro\subseteq B(H,K)$ is a norm-closed TRO such that
    \[
    \ti=[\tro\es\tro^*]^\normcl \ \ \text{and} \ \ \es=[\tro^*\ti\tro]^\normcl.
    \]
    A standard argument (see \cite[Lemma 2.3]{hereditary}, \cite[Theorem 6.1]{bmp} and \cite[Lemma 3.4]{ele}) proves that $\tro$ is a $\sigma$-TRO, thus Theorem \ref{opspc} applies with $H_1=H_2=H$, $K_1=K_2=K$, $\tro_1=\tro_2=\tro$, $\emm_1=\emm_2$ and $\en_1=\en_2$, so that $U_1=U_2=:U$.
\end{proof}
The above constitutes an alternative proof for the equivalence (i) $\Leftrightarrow$ (iv) of \cite[Theorem 3.8]{systems}, which we state separately as a Corollary:
\begin{corollary}
    Two operator systems $\es$ and $\ti$ are strongly $\Delta$-equivalent (as operator systems) if and only if they are stably isomorphic (in the sense that $\es\otimes\kk\cong\ti\otimes\kk$ via a surjective complete order complete isometry).
\end{corollary}

\subsection*{Acknowledgments}

The author would like to thank G. K. Eleftherakis for his great guidance and useful suggestions for this note. The research described in this paper was carried out within the framework of the National Recovery and Resilience Plan Greece 2.0, funded by the European Union - NextGenerationEU (Implementation Body: HFRI. Project name: NONCOMMUTATIVE ANALYSIS: OPERATOR SYSTEMS AND NONLOCALITY. HFRI Project Number: 015825).

\end{document}